\documentclass[submission,copyright,creativecommons]{eptcs}
\providecommand{\event}{FROM 2019} 
\usepackage{breakurl}             
\usepackage{underscore}           

\usepackage{doi}
\usepackage{amsmath}
\usepackage{graphicx}
\usepackage{amssymb}
\usepackage{graphics} 
\usepackage{latexsym}
\newcommand\new[0]{\reflectbox{\ensuremath{\mathsf{N}}}} 

\newtheorem{theorem}{Theorem}[section]
\newtheorem{proposition}[theorem]{Proposition}

\newtheorem{definition}[theorem]{Definition}

\newtheorem{remark}[theorem]{Remark}
\newenvironment{proof}[1][Proof]{\noindent\textbf{#1.} }{\ \rule{0.5em}{0.5em}}

\title{Finitely Supported Sets Containing \\ Infinite Uniformly Supported Subsets}
\author{Andrei Alexandru
\institute{Romanian Academy, Institute of Computer Science, Ia\c si, Romania}
\email{andrei.alexandru@iit.academiaromana-is.ro}
\and
Gabriel Ciobanu 
\institute{Romanian Academy and A.I.Cuza University, Ia\c{s}i, Romania}
\email{gabriel@info.uaic.ro}
}
\def\titlerunning {Finitely Supported Sets Containing Infinite Uniformly Supported Subsets}
\def\authorrunning{A. Alexandru and G. Ciobanu}

\begin{document}
\maketitle

\begin{abstract}

The theory of finitely supported algebraic structures represents a 
reformulation of Zermelo-Fraenkel set theory in which every construction is 
finitely supported according to the action of a group of permutations of 
some basic elements named atoms. In this paper we study the properties of 
finitely supported sets that contain infinite uniformly supported subsets, 
as well as the properties of finitely supported sets that do not contain 
infinite uniformly supported subsets. For classical atomic sets, we study 
whether they contain or not infinite uniformly supported subsets.

\end{abstract}

\section{Finitely Supported Sets}

Finitely supported mathematics \cite{book} is dealing with the set 
theory foundations for the finitely supported structures. 
Finitely supported structures are related to the recent development of the 
Fraenkel-Mostowski axiomatic set theory working with `nominal sets' and 
dealing with binding and fresh names in computer science \cite{pitts-2}, 
but also to the theory of admissible sets of Barwise \cite{barwise}, in particular 
to the theory of hereditary finite sets. Fraenkel-Mostowski set theory (FM) 
represents an axiomatization of the Fraenkel Basic Model for the 
Zermelo-Fraenkel set theory with atoms (ZFA), a model used originally to prove 
the independence of the axiom of choice and other axioms of set theory 
with atoms. Nominal sets are actually a Zermelo-Fraenkel set theory (ZF) 
alternative to the non-standard Fraenkel-Mostowski set theory whose axioms are 
the ZFA axioms together with a new axiom of finite support claiming that any 
set-theoretical construction has to be finitely supported modulo a canonical 
hierarchically defined permutation action), since nominal sets are 
defined by involving group actions over standard ZF sets, without being 
necessary to modify the ZF axioms of extensionality or foundation.
A nominal set is defined as a usual Zermelo-Fraenkel set endowed with 
a group action of the group of (finitary) permutations over a certain 
fixed countable ZF set $A$ of basic elements whose internal structure is 
ignored (called \emph{atoms}), satisfying also a finite support requirement. 
This finite support requirement states that for any element in a nominal 
set there should exist a finite set of atoms such that any permutation 
fixing pointwise this set of atoms also leaves the element invariant under 
the related group action. By now, nominal sets were used to study the binding, 
scope, freshness and renaming in programming languages and related formal 
systems. The inductively defined finitely supported sets (that are finitely 
supported elements in the powerset of a nominal set) involving the 
name-abstraction together with Cartesian product and disjoint union can 
encode formal syntax modulo renaming of bound variables. In this way, the 
standard theory of algebraic data types can be extended to include 
signatures involving binding operators. In particular, there exists an 
associated notion of structural recursion for defining syntax-manipulating 
functions and a notion of proof by structural induction. 
Certain generalizations of nominal sets are involved in the study of 
automata, programming languages or Turing machines over infinite 
alphabets; for this, a relaxed notion of finiteness called `orbit 
finiteness' was defined; it means `having a finite number of orbits 
(equivalence classes) under a certain group action'~\cite{boj}. 
Fraenkel-Mostowski generalized set theory (FMG) was introduced in~\cite{gabb} 
and generalizes both the size of atoms and the size of support from the 
FM set theory. More exactly, it is presented a generalization of 
the FM sets by replacing `finite support' with `well-orderable (at least 
countable) support' and by considering an uncountable set of atoms. Notions 
such as abstraction and freshness quantifier $\new$ in the FM set theory 
have also been extended into the new framework. In this sense, 
in FMG $\new a.p(a)$ for a predicate $p$ means that $p$ holds for all atoms 
except a well-orderable subset of atoms, while in FM $\new a.p(a)$ means 
that $p$ holds for all atoms except a finite subset of atoms. This approach 
allows binding of infinitely many names in syntax instead of only finitely 
many names.
A very recent work describing a general framework for reasoning about 
syntax with bindings is \cite{popescu}; it overlaps the nominal sets 
framework, but also provides significant distinctions. In this paper, 
the authors employed functors for modelling the presence of variables 
instead of sets with atoms. Furthermore, the authors are able to remove 
the finite support restriction and to accept terms that are infinitely 
branching, terms having infinite depth, or both. Unlike nominal sets 
theory where atoms can only be manipulated via bijections, the functors 
described in \cite{popescu} distinguish between binding variables 
(managed via bijections) and free variables (managed via possibly 
non-bijective functions); these functors allow the authors to apply not 
only swappings or permutations, but also arbitrary substitutions.

Finitely supported mathematics (shortly, FSM) is focused on the foundations 
of set theory (rather than on applications in computer science). In order 
to describe FSM as a theory of finitely supported algebraic structures, we 
refer to the theory of nominal sets (with the mention that the requirement 
regarding the countability of $A$ is irrelevant). We call these sets 
\emph{invariant sets}, using the motivation of Tarski regarding logicality 
(more precisely, a logical notion is defined by Tarski as one that is 
invariant under the permutations of the universe of discourse). FSM is 
actually represented by finitely supported subsets of invariant sets 
together with finitely supported internal algebraic operations or with 
finitely supported relations (that should be finitely supported as subsets 
in the Cartesian product of two invariant sets). There is no major 
technical difference between `FSM' and `nominal' (related to basic 
definitions), but conceptually the nominal approach is related to computer 
science, while FSM deals with the foundations of mathematics (and 
experimental sciences) by studying the consistency and inconsistency of 
various results within the framework of the atomic sets. Our goal is not to 
re-brand the nominal framework (whose value we certainly recognize), but 
to provide a collection of set theoretical results regarding foundations of 
finitely supported structures.

FSM contains both the family of `non-atomic' (i.e., ordinary) ZF sets 
which are proved to be trivial FSM sets (i.e., their elements are left 
unchanged under the effect of the canonical permutation action) and the 
family of `atomic' sets (i.e., sets that contain at least an atom 
somewhere in their structure) with finite supports (hierarchically 
constructed from the empty set and the fixed ZF set $A$ of atoms). One 
task is to analyze whether a classical ZF result (obtained in the 
framework of non-atomic sets) can be adequately reformulated by replacing 
`non-atomic ZF element/set/structure' with `atomic and finitely supported 
element/set/structure' in order to be valid also for atomic sets with 
finite supports.

Note that the FSM sets is not closed under ZF subsets constructions, 
meaning that there exist subsets of FSM sets that fail to be finitely 
supported (for example the simultaneously ZF infinite and coinfinite 
subsets of the set $A$). Thus, for proving results in FSM we cannot use 
related results from the ZF framework without reformulating them with 
respect to the finite support requirement. Furthermore, not even the 
translation of the results from a non-atomic framework into an atomic 
framework (such as Zermelo Fraenkel set theory with atoms obtained by 
weakening ZF axiom of extensionality) is an easy task. Results from ZF may 
lose their validity when reformulating them in Zermelo Fraenkel set theory 
with atoms. For example, it is known that multiple choice principle and 
Kurepa's maximal antichain principle are both equivalent to the axiom of 
choice in ZF. However, Jech proved in \cite{jech} that multiple choice 
principle is valid in the Fraenkel Second Model, while the axiom of choice 
fails in this model. Furthermore, Kurepa's maximal antichain principle is 
valid in the Fraenkel Basic Model, while the axiom of choice fails in this 
model. This means that the following two statements that are valid in ZF, 
namely \emph{`Kurepa's principle implies axiom of choice'} and 
\emph{`Multiple choice principle implies axiom of choice'} fail in 
Zermelo Fraenkel set theory with atoms.

A proof of an FSM result should be internally consistent in FSM and not 
retrieved from ZF, that is it should involve only finitely supported 
constructions (even in the intermediate steps). The meta-theoretical 
techniques for the translation of a result from non-atomic structures to 
atomic structures are based on a refinement of the finite support 
principle from \cite{pitts-2}, a refinement called `$S$-finite supports 
principle' claiming that {\it{for any finite set $S$ of atoms, anything that is 
definable in higher-order logic from $S$-supported structures by using 
$S$-supported constructions is also $S$-supported}}. The formal involvement 
of the $S$-finite support principles actually implies a hierarchical 
constructive method for defining the support of a structure by employing, 
step-by-step, the supports of the substructures of a related structure.

\section{Preliminary Results}

A finite set is a set of the form $\{x_{1}, \ldots, x_{n}\}$. 
Consider a fixed ZF infinite set~$A$ of elements that can be checked only 
for equality. The elements of $A$ are called 'atoms' by analogy with the 
models of the classic ZFA set theory given by Fraenkel and 
Mostowski. A \emph{transposition} is a function $(a\, b):A\to A$ that 
interchanges only $a$ and $b$. A \emph{(finitary) permutation} of $A$ in 
FSM is a bijection of $A$ generated by composing finitely many 
transpositions. We denote by~$S_{A}$ the group of all (finitary) 
permutations of $A$. According to Proposition 2.6 in \cite{book}, a 
bijection on~$A$ is finitely supported if and only if it is a (finitary) 
permutation of $A$. Thus, (finitary) permutations are simply called 
permutations.

\begin{definition}\label{2.4} \ 
\begin{enumerate}
\item Let $X$ be a ZF set. An \emph{$S_{A}$-action} on $X$ is a group 
action $\cdot$ of $S_{A}$ on $X$. An \emph{$S_{A}$-set} is a pair 
$(X,\cdot)$, where $X$ is a ZF set, and $\cdot$ is an $S_{A}$-action on $X$.

\item Let $(X,\cdot)$ be an $S_{A}$-set. We say that \emph{$S\subset A$ 
supports $x$} whenever for each $\pi\in Fix(S)$ we have $\pi\cdot x=x$, 
where $Fix(S)=\{\pi\,|\,\pi(a)=a,\forall a\in S\}$. The least finite set 
(w.r.t. the inclusion relation) supporting $x$ (which exists according to 
\cite{book}) is called \emph{the support of $x$} and is denoted by 
$supp(x)$. An empty supported element is called \emph{equivariant}.

\item Let $(X,\cdot)$ be an $S_{A}$-set. We say that $X$ is an 
\emph{invariant set} if for each $x\in X$ there exists a finite set 
\emph{$S_{x}\subset A$ }which supports $x$. 
\end{enumerate}
\end{definition}

\begin{proposition}  \label{p1} \cite{book,pitts-2}
 Let $(X,\cdot)$ and $(Y,\diamond)$ be $S_{A}$-sets.
\begin{enumerate}
\item The set $A$ of atoms is an invariant set with the $S_{A}$-action
$\cdot:S_{A}\times A\rightarrow A$ defined by $\pi\cdot a:=\pi(a)$
for all $\pi\in S_{A}$ and $a\in A$.  Furthermore, $supp(a)=\{a\}$
for each $a\in A$. 

\item Let $\pi\in S_{A}$. If $x\in X$ is 
finitely supported, then $\pi\cdot x$ is finitely supported and 
$supp(\pi\cdot x)=\{\pi(u)\,|\,u \in supp(x)\}:=\pi(supp(x))$.

\item The Cartesian 
product $X\times Y$ is also an $S_{A}$-set with the $S_{A}$-action 
$\otimes:S_{A}\times(X\times Y)\rightarrow(X\times Y)$ defined by 
$\pi\otimes(x,y)=(\pi\cdot x,\pi\diamond y)$ for all $\pi\in S_{A}$ and all 
$x\in X$, $y\in Y$. If $(X,\cdot)$ and $(Y,\diamond)$ are invariant sets, 
then $(X\times Y,\otimes)$ is also an invariant set.

\item The powerset $\wp(X)=\{Z\,|\, 
Z\subseteq X\}$ is also an $S_{A}$-set with the $S_{A}$-action $\star: 
S_{A}\times\wp(X) \rightarrow \wp(X)$ defined by $\pi\star Z:=\{\pi\cdot 
z\,|\, z\in Z\}$ for all $\pi \in S_{A}$, and all $Z \subseteq X$. For 
each invariant set $(X,\cdot)$, we denote by $\wp_{fs}(X)$ the set of elements in $\wp(X)$ which are finitely supported according to the 
action~$\star$ .
$(\wp_{fs}(X),\star|_{\wp_{fs}(X)})$ is an invariant set.

\item The finite powerset of $X$ denoted by $\wp_{fin}(X)=\{Y \subseteq 
X\,|\, Y \text{finite}\}$ and the cofinite powerset of $X$ denoted by 
$\wp_{cofin}(X)=\{Y \subseteq X\,|\, X\setminus Y \text{finite}\}$ are 
both $S_{A}$-sets with the $S_{A}$-action $\star$ defined as in the 
previous item~(2). If $X$ is an invariant set, then both $\wp_{fin}(X)$ and 
$\wp_{cofin}(X)$ are invariant sets.

\item We have $\wp_{fs}(A)=\wp_{fin}(A) \cup \wp_{cofin}(A)$. \ 
If $X \in \wp_{fin}(A)$, then $supp(X)=X$.\\ 
If $X \in \wp_{cofin}(A)$, then $supp(X)=A \setminus X$.

\item The
disjoint union of $X$ and $Y$ defined by $X+Y=\{(0,x)\,|\, x\in X\}\cup\{(1,y)\,|\, y\in Y\}$
is an $S_{A}$-set with the $S_{A}$-action $\star:S_{A}\times(X+Y)\rightarrow(X+Y)$
defined by $\pi\star z=(0,\pi\cdot x)$ if $z=(0,x)$ and $\pi\star z=(1,\pi\diamond y)$
if $z=(1,y)$. If $(X,\cdot)$ and $(Y,\diamond)$ are invariant sets, then
$(X+Y,\star)$ is also an invariant set.

\item Any ordinary (non-atomic) ZF-set $X$ (such as $\mathbb{N},\mathbb{Z},\mathbb{Q}$
or $\mathbb{R}$ for example) is an invariant set with the single possible $S_{A}$-action
$\cdot:S_{A}\times X\rightarrow X$ defined by $\pi\cdot x:=x$ for
all $\pi \in S_{A}$ and $x\in X$. 
\end{enumerate}
\end{proposition}

\begin{definition}\label{2.14} \ 
\begin{enumerate}
\item Let $(X,\cdot)$ be an $S_{A}$-set. A subset~$Z$ of $X$ is called 
\emph{finitely supported} if and only if $Z\in\wp_{fs}(X)$. 
 A subset $Z$ of $X$ is \emph{uniformly supported} if all the elements 
of~$Z$ are supported by the same set $S$ (and so $Z$ is itself supported 
by $S$).

\item Let $(X,\cdot)$ be a finitely supported subset of an $S_{A}$- set 
$(Y, \cdot)$. A subset~$Z$ of $Y$ is called \emph{finitely supported 
subset of $X$} (and we denote this by $Z \in \wp_{fs}(X)$) if and only if 
$Z\in\wp_{fs}(Y)$ and $Z \subseteq X$. Similarly, we say that a uniformly 
supported subset of $Y$ contained in $X$ is a \emph{uniformly supported 
subset of $X$}.
\end{enumerate}
\end{definition} 

From Definition \ref{2.4}, a subset~$Z$ of an invariant set $(X, \cdot)$ 
is finitely supported by a set $S \subseteq A$ if and only if $\pi \star Z 
\subseteq Z$ for all $\pi \in Fix(S)$, i.e. if and only if $\pi \cdot z 
\in Z$ for all $\pi \in S_{A}$ and all $z \in Z$. This is because any 
permutation of atoms should have finite order, and so the relation $\pi \star Z 
\subseteq Z$ is equivalent to $\pi \star Z 
= Z$.

Due to Proposition \ref{p1}(2), whenever $X$ is a finitely supported 
subset of an invariant set~$Y$, the uniform powerset of $X$ 
denoted by $\wp_{us}(X)=\{Z\!\subseteq\! X\,|\, Z\, \text{uniformly 
supported}\}$ is a subset of $\wp_{fs}(Y)$ supported by $supp(X)$. This is 
because, whenever $Z \subseteq X$ is uniformly supported by $S$ and $\pi 
\in Fix(supp(X))$, we have $\pi \star Z \subseteq \pi \star X=X$ and $\pi 
\star Z$ is uniformly supported by $\pi (S)$. Similarly, $\wp_{fin}(X)$ 
and $\wp_{cofin}(X)$ are subsets of $\wp_{fs}(Y)$ supported by $supp(X)$. 
We consider that $\emptyset$, being a finite subset of $X$, belongs to 
$\wp_{us}(X)$.

\begin{definition}\label{2.10-1}
Let $X$ and $Y$ be invariant sets.
\begin{enumerate} 
\item A function $f:X\rightarrow Y$ is \emph{finitely supported} if 
$f\in\wp_{fs}(X\times Y)$. 
\item Let $Z$ be a finitely supported subset of $X$ and $T$ a finitely 
supported subset of $Y$. A function $f:Z\rightarrow T$ is \emph{finitely 
supported} if $f\in\wp_{fs}(X\times Y)$. The set of all finitely supported 
functions from $Z$ to $T$ is denoted by $T^{Z}_{fs}$.
\end{enumerate}
\end{definition}

\begin{proposition}\label{2.18'} \cite{book, pitts-2} 
Let $(X,\cdot)$ and $(Y,\diamond)$ be two invariant sets.
\begin{enumerate}
\item   $Y^{X}$ (i.e. the set of all functions from $X$ to $Y$) is an 
$S_{A}$-set with the $S_{A}$-action $\widetilde{\star}:S_{A}\times 
Y^{X}\rightarrow Y^{X}$ defined by $(\pi \widetilde{\star}f)(x) = 
\pi\diamond(f(\pi^{-1}\cdot x))$ for all $\pi\in S_{A}$, $f\in Y^{X}$ and 
$x\in X$. A function $f:X\rightarrow Y$ is finitely supported (in the sense 
of Definition \ref{2.10-1}) if and only if it is finitely supported with 
respect the permutation action $\widetilde{\star}$.
\item Let $Z$ be a finitely supported subset of $X$ and $T$ a finitely 
supported subset of $Y$. A function $f:Z\rightarrow T$ is supported by 
a finite set $S \subseteq A$ if and only if for all $x \in Z$ and all 
$\pi \in Fix(S)$ we have $\pi \cdot x \in Z$, $\pi \diamond f(x) \in T$ 
and $f(\pi\cdot x)=\pi\diamond f(x)$. 
\end{enumerate} 
\end{proposition}

\section{FSM Uniformly Infinite Sets} \label{chap9}

\begin{definition} \label{dddd}
Let $X$ be a finitely supported subset of an invariant set~$Y$. $X$ is 
called \emph{FSM uniformly infinite} if there exists an infinite, 
uniformly supported subset of $X$. 
Otherwise, we call $X$ \emph{FSM non-uniformly infinite}. 
\end{definition}

\begin{theorem} \label{lem1} 
Let $X$ be a finitely supported subset of an invariant set $(Y, \cdot)$ 
such that $X$ is not FSM uniformly infinite.  Then the set $\wp_{us}(X)$ 
is not FSM uniformly infinite. 
\end{theorem}

\begin{proof}Suppose, by contradiction, that the set $\wp_{us}(X)$ contains an infinite 
subset $\mathcal{F}$ such that all the elements of $\mathcal{F}$ are 
different and supported by the same finite set $S$. 
By convention, without assuming that $i \mapsto X_{i}$ is finitely 
supported, we understand $\mathcal{F}$ as $\mathcal{F}=(X_{i})_{i \in I}$ with 
the properties that $X_{i} \neq X_{j}$ whenever $i \neq j$ and 
$supp(X_{i}) \subseteq S$ for all $i \in I$. Let us fix an arbitrary $j \in I$. 
We prove that $supp(X_{j})=\underset{x \in X_{j}}{\cup}supp(x)$. Indeed, 
since $X_{j}$ is uniformly supported, there exists a finite subset of atoms 
$T$ such that~$T$ supports every $x \in X_{j}$, i.e. $supp(x) \subseteq T$ for 
all $x \in X_{j}$. Thus, $\cup\{supp(x)\,|\, x\in X_{j}\} \subseteq T$. Clearly, 
$supp(X_{j}) \subseteq \cup\{supp(x)\,|\, x\in X_{j}\}$. Conversely, let $a \in 
\cup\{supp(x)\,|\, x\in X_{j}\}$. Thus, there exists $x_{0} \in X_{j}$ such that 
$a \in supp(x_{0})$. Let $b$ be an atom such that $b \notin supp(X_{j})$ and 
$b \notin T$. Such an atom exists because~$A$ is infinite, while $supp(X_{j})$ 
and $T$ are both finite. We prove by contradiction that $(b\; a) \cdot 
x_{0} \notin X_{j}$. Indeed, suppose that $(b\; a) \cdot x_{0}=y \in X_{j}$.  
Since $a \in supp(x_{0})$, by Proposition \ref{p1}(2), we have 
$b=(b\;a)(a) \in (b\; a)(supp(x_{0}))=supp((b\; a) \cdot x_{0})=supp(y)$. 
Since $supp(y) \subseteq T$, we get $b \in T$: a contradiction! Therefore, 
$(b\; a) \star X_{j} \neq X_{j}$, where~$\star$ is the canonical $S_{A}$-action on 
$\wp(Y)$.  Since $b \notin supp(X_{j})$, we prove by contradiction that $a \in 
supp(X_{j})$. Indeed, suppose that $a \notin supp(X_{j})$. We have that $(b\; a) 
\in Fix(supp(X_{j}))$. Since $supp(X_{j})$ supports $X_{j}$, it follows that $(b\; a) 
\star X_{j}=X_{j}$ which is a contradiction. Thus, $a \in supp(X_{j})$ and so 
$supp(X_{j})=\underset{x \in X_{j}}{\cup}supp(x)$.

Therefore, because $supp(X_{j}) \subseteq S$,~$X_{j}$ has 
the property that $supp(x)\subseteq S$ for all $x \in X_{j}$. Since $j$ 
has been arbitrarily chosen from $I$, it follows that $\underset{ i \in 
I}{\cup}X_{i}$ is an uniformly supported subset of $X$ (all its elements 
being supported by $S$). Furthermore, $\underset{i \in I}{\cup}X_{i}$ is 
infinite since the family $(X_{i})_{i \in I}$ is infinite and $X_{i} \neq 
X_{j}$ whenever $i \neq j$. This contradicts the hypothesis. 
\end{proof}

\begin{theorem} \label{lem2}
Let $X$ be a finitely supported subset of an invariant set $(Y, \cdot)$ 
such that $X$ is not FSM uniformly infinite. 
Then the set $\wp_{fin}(X)$ is not FSM uniformly infinite. 
\end{theorem}
\begin{proof} 
We always have that $\wp_{fin}(X) \subseteq \wp_{us}(X)$ because any 
finite subset of $X$ of form $\{x_{1}, \ldots, x_{n}\}$ is uniformly 
supported by $supp(x_{1})\cup\ldots\cup supp(x_{n})$. Since $\wp_{us}(X)$ 
does not contain an infinite uniformly supported subset, it follows that 
neither $\wp_{fin}(X)$ contains an infinite uniformly supported subset. 
\end{proof}

\begin{theorem}\label{th1} 
Let $X$ be a finitely supported subset of an invariant set $(Y, \cdot)$. 
\begin{enumerate}
\item If $X$ is not FSM uniformly infinite, then any finitely supported 
order-preserving (with respect to the inclusion relation) function 
$f:\wp_{us}(X)\to \wp_{us}(X)$ has a least fixed point supported 
by $supp(f) \cup supp(X)$.

\item If $X$ is not FSM uniformly infinite, then any finitely supported 
order-preserving (with respect to the inclusion relation) function 
$f:\wp_{fin}(X)\to \wp_{fin}(X)$ has a least fixed point supported by 
$supp(f) \cup supp(X)$.
\end{enumerate}
\end{theorem}

\begin{proof}Let $f:\wp_{us}(X)\to \wp_{us}(X)$ be a finitely supported 
order-preserving function. Firstly, since $\wp_{us}(X)$ is a subset of 
$\wp_{fs}(Y)$ supported by $supp(X)$, we have $\pi \star \emptyset, 
\pi^{-1} \star \emptyset \in \wp_{us}(X)$ for any permutation $\pi \in 
Fix(supp(X))$. Thus, $\emptyset \subseteq \pi \star 
\emptyset$ and $\emptyset \subseteq \pi^{-1} \star \emptyset$. Since the 
relation $\subseteq$ on $\wp_{us}(X)$ is supported by $supp(X)$, we get 
$\pi \star \emptyset \subseteq \pi \star (\pi^{-1} \star \emptyset)= (\pi 
\circ \pi ^{-1}) \star \emptyset= \emptyset$, and so $\emptyset = \pi 
\cdot \emptyset$ which means that $\emptyset$ is an element in 
$\wp_{us}(X)$ supported by $supp(X)$. Actually, $\emptyset$ belongs to 
$\wp_{fin}(X)$ that is a subset of $\wp_{us}(X)$.

Since $\emptyset \subseteq f(\emptyset)$ and $f$ is order-preserving, we 
can define the ascending sequence $\emptyset \subseteq f(\emptyset) 
\subseteq f^{2}(\emptyset) \subseteq \ldots \subseteq f^{n}(\emptyset) 
\subseteq \ldots$, where $f^{n}(\emptyset)=f(f^{n-1}(\emptyset))$ and 
$f^{0}(\emptyset)=\emptyset$. We prove by induction that 
$(f^{n}(\emptyset))_{n \in \mathbb{N}}$ is uniformly supported by $supp(f) 
\cup supp(X)$, namely $supp(f^{n}(\emptyset))\subseteq supp(f) \cup 
supp(X)$ for each $n \in \mathbb{N}$.
We have $supp(f^{0}(\emptyset))=supp(\emptyset) \subseteq supp(X) 
\subseteq supp(f) \cup supp(X)$. Let us assume that 
$supp(f^{n}(\emptyset))\subseteq supp(f) \cup supp(X)$ for some $n \in 
\mathbb{N}$. We have to prove that $supp(f^{n+1}(\emptyset))\subseteq 
supp(f) \cup supp(X)$. Let $\pi \in Fix(supp(f) \cup supp(X))$. From the 
inductive hypothesis, we have $\pi \in Fix(supp(f^{n}(\emptyset)))$ and 
so $\pi\star f^{n}(\emptyset)=f^{n}(\emptyset)$. Since $\pi$ fixes 
$supp(f)$ pointwise, according to Proposition~\ref{2.18'}, we have 
$\pi\star f^{n+1}(\emptyset)=\pi\star f(f^{n}(\emptyset))=f(\pi\star 
f^{n}(\emptyset))=f(f^{n}(\emptyset))=f^{n+1}(\emptyset)$. Therefore, 
$(f^{n}(\emptyset))_{n \in \mathbb{N}} \subseteq \wp_{us}(X)$ is uniformly 
supported by $supp(f) \cup supp(X)$.  Thus, according to Theorem \ref{lem1}, 
$(f^{n}(\emptyset))_{n \in \mathbb{N}}$ should be finite, and so there 
exists $n_{0} \in \mathbb{N}$ such that 
$f^{n}(\emptyset)=f^{n_{0}}(\emptyset)$ for all $n \geq n_{0}$. Thus, 
$f(f^{n_{0}}(\emptyset))=f^{n_{0}+1}(\emptyset)=f^{n_{0}}(\emptyset)$, and 
so $f^{n_{0}}(\emptyset)$ is a fixed point of $f$. It is supported by 
$supp(f) \cup supp(X)$, and obviously it is the least one.

2. A similar argument allows us to prove the second item of the 
proposition. This time Theorem~\ref{lem2} is used to prove that 
the uniformly supported ascending family $(f^{n}(\emptyset))_{n \in 
\mathbb{N}} \subseteq \wp_{fin}(X)$ is finite, and so it is stationary. 
\end{proof}

\begin{theorem} \label{th11} 
Let $X$ be a finitely supported subset of an invariant set $(Y, \cdot)$. 
\begin{enumerate}
\item If $X$ is not FSM uniformly infinite and $f:\wp_{us}(X)\to 
\wp_{us}(X)$ is finitely supported with the property that $Z \subseteq 
f(Z)$ for all $Z \in \wp_{us}(X)$, then for each $Z \in \wp_{us}(X)$ there 
exists some $m \in \mathbb{N}$ such that $f^{m}(Z)$ is a fixed point of $f$.
\item If $X$ is not FSM uniformly infinite and $f:\wp_{fin}(X)\to 
\wp_{fin}(X)$ is finitely supported with the property that $Z \subseteq 
f(Z)$ for all $Z \in \wp_{fin}(X)$, then for each $Z \in \wp_{fin}(X)$ there 
exists some $m \in \mathbb{N}$ such that $f^{m}(Z)$ is a fixed point of $f$.
\end{enumerate}
\end{theorem}

\begin{proof} 1. Let us fix an arbitrary element $Z \in \wp_{us}(X)$. We 
consider the ascending (via sets inclusion) sequence $(Z_{n})_{n \in 
\mathbb{N}}$ which has the first term $Z_{0}=Z$ and the general term 
$Z_{n+1}=f(Z_{n})$ for all $n \in \mathbb{N}$.  We prove by induction that 
$supp(Z_{n}) \subseteq supp(f)\cup supp(Z) \cup supp(X)$ for all $n \in 
\mathbb{N}$. Clearly, $supp(Z_{0})=supp(Z) \subseteq supp(f)\cup supp(Z) 
\cup supp(X)$. Assume that $supp(Z_{k}) \subseteq supp(f)\cup supp(Z) \cup 
supp(X)$. Let $\pi \in Fix(supp(f)\cup supp(Z) \cup supp(X))$. Thus, $\pi 
\cdot Z_{k}=Z_{k}$ according to the inductive hypothesis. According to 
Proposition~\ref{2.18'}, because $\pi$ fixes $supp(f)$ pointwise, $supp(f)$ 
supports $f$ and $\wp_{us}(X)$ is supported by $supp(X)$, we get $\pi 
\star Z_{k+1}= \pi \star f(Z_{k})=f(\pi \star Z_{k})=f(Z_{k})=Z_{k+1}$. 
Since $supp(Z_{k+1})$ is the least set supporting $Z_{k+1}$, we obtain 
$supp(Z_{k+1}) \subseteq supp(f)\cup supp(Z) \cup supp(X)$. Thus, 
$(Z_{n})_{n \in \mathbb{N}} \subseteq \wp_{us}(X)$ is uniformly supported 
by $supp(f) \cup supp(Z) \cup supp(X)$, and so $(Z_{n})_{n \in \mathbb{N}}$ 
must be finite according to Theorem \ref{lem1}. Since by hypothesis we 
have $Z_{0} \subseteq Z_{1} \subseteq \ldots \subseteq Z_{n} \subseteq 
\ldots$, there should exist $m \in \mathbb{N}$ such that $Z_{m}=Z_{m+1}$, 
i.e. $f^{m}(Z)=f^{m+1}(Z)=f(f^{m}(Z))$, and so the result follows.

2. A similar argument allows us to prove the second item of this theorem. 
Theorem~\ref{lem2} is used to prove that the uniformly supported 
ascending family $(f^{n}(Z))_{n \in \mathbb{N}} \subseteq \wp_{fin}(X)$ 
is finite, and so it is stationary for every $Z \in \wp_{fin}(X)$.
\end{proof}

For self-mappings on $\wp_{fin}(A)$ we have the following stronger property.  
\begin{proposition} \label{pf1ty} 
Let $f:\wp_{fin}(A)\to \wp_{fin}(A)$ be a finitely supported function with 
the property that $Z \subseteq f(Z)$ for all $Z \in \wp_{fin}(A)$. There 
are infinitely many fixed points of $f$, namely the finite subsets of $A$ 
containing all the elements of $supp(f)$.
\end{proposition}

\begin{proof}
Let $Z \in \wp_{fin}(A)$. Since the support of a finite subset of atoms 
coincides with the related subset, we have $supp(Z)=Z$ and 
$supp(f(Z))=f(Z)$. According to Proposition \ref{2.18'}, for any 
permutation $\pi \in Fix(supp(f) \cup supp(Z))=Fix(supp(f) \cup Z)$, we have 
$\pi \star f(Z)=f(\pi \star Z)=f(Z)$ which means $supp(f) \cup Z$ supports 
$f(Z)$, that is, $f(Z)=supp(f(Z)) \subseteq supp(f) \cup Z$ (claim 1). 
Since we also have $Z \subseteq f(Z)$, we get $Z\setminus supp(f) \subseteq 
f(Z) \setminus supp(f) \subseteq Z \setminus supp(f)$, that is, $Z 
\setminus supp(f)=f(Z) \setminus supp(f)$ (claim 2). If 
$supp(f)=\emptyset$, the result follows obviously. Let 
$supp(f)=\{a_{1},\ldots,a_{n}\}$. According to (claim 1), we have $supp(f) 
\subseteq f(supp(f)) \subseteq supp(f)$, and so $f(supp(f))=supp(f)$. If 
$Z$ has the form $Z=\{a_{1},\ldots, a_{n},b_{1},\ldots,b_{m}\}$ with 
$b_{1},\ldots,b_{m} \in A \setminus supp(f)$, $m \geq 1$, we should have by 
hypothesis that $a_{1},\ldots,a_{n} \in f(Z)$, and by (claim 2) 
$f(Z)\setminus supp(f)=\{b_{1},\ldots, b_{m}\}$. Since no other elements 
different from $a_{1},\ldots,a_{n}$ are in $supp(f)$, from (claim 1) we get 
$f(Z)=\{a_{1},\ldots, a_{n},b_{1},\ldots,b_{m}\}$. 
\end{proof}

\begin{theorem} \label{ti1} 
The following properties of FSM uniformly infinite sets hold.
\begin{enumerate}
\item Let $X$ be an infinite, finitely supported subset of an invariant set $Y$. Then the sets $\wp_{fs}(\wp_{fin}(X))$ and $\wp_{fs}(T_{fin}(X))$ are FSM uniformly infinite.
\item Let $X$ be an infinite, finitely supported subset of an invariant set $Y$. Then the set $\wp_{fs}(\wp_{fs}(X))$ is FSM uniformly infinite.
\item Let $X$ and $Y$ be two finitely supported subsets of an invariant set $Z$. If neither $X$ nor $Y$ is FSM uniformly infinite, then $X \times Y$ is not FSM uniformly infinite. 
\item Let $X$ and $Y$ be two finitely supported subsets of an invariant set $Z$. If neither $X$ nor $Y$ is FSM uniformly infinite, then $X + Y$ is not FSM uniformly infinite. 
\end{enumerate}
\end{theorem}
\begin{proof}
1. Obviously, $\wp_{fin}(X)$ is a finitely supported subset of the invariant 
set $\wp_{fs}(Y)$, supported by $supp(X)$. This is because whenever $Z$ is 
an element of $\wp_{fin}(X)$ (i.e. whenever $Z$ is a finite subset of $X$) 
and $\pi$ fixes $supp(X)$ pointwise, we have that $\pi \star Z$ is also a 
finite subset of $X$. The family $\wp_{fs}(\wp_{fin}(X))$ represents the 
family of those subsets of $\wp_{fin}(X)$ which are finitely supported as 
subsets of the invariant set $\wp_{fs}(Y)$ in the sense of Definition 
\ref{2.14}. As above, according to Proposition \ref{p1}, we have that 
$\wp_{fs}(\wp_{fin}(X))$ is a finitely supported subset of the invariant 
set $\wp_{fs}(\wp_{fs}(Y))$, supported by $supp(\wp_{fin}(X)) \subseteq 
supp(X)$.

Let $X_{i}$ be the set of all $i$-sized subsets from $X$, i.e. $X_{i}=\{Z 
\subseteq X\,|\,|Z|=i\}$. Since $X$ is infinite, it follows that each 
$X_{i}, i \geq 1$ is non-empty. Obviously, we have that any $i$-sized 
subset $\{x_{1}, \ldots, x_{i}\}$ of $X$ is finitely supported (as a 
subset of $Y$) by $supp(x_{1}) \cup \ldots \cup supp(x_{i})$. Therefore, 
$X_{i} \subseteq \wp_{fin}(X)$ and $X_{i} \subseteq \wp_{fs}(Y)$ for all 
$i \in \mathbb{N}$. Since $\cdot$ is a group action, the image of an 
$i$-sized subset of $X$ under an arbitrary permutation is an $i$-sized 
subset of $Y$. However, any permutation of atoms that fixes $supp(X)$ 
pointwise also leaves~$X$ invariant, and so for any permutation $\pi \in 
Fix(supp(X))$ we have that $\pi \star Z$ is an $i$-sized subset of $X$ 
whenever $Z$ is an $i$-sized subset of $X$. Thus, each $X_{i}$ is a subset 
of $\wp_{fin}(X)$ finitely supported by $supp(X)$, and so $X_{i} \in 
\wp_{fs}(\wp_{fin}(X))$. The family $(X_{i})_{i \in \mathbb{N}}$ is 
infinite and uniformly supported.

If we consider $Y_{i}$ the set of all $i$-sized injective tuples formed by 
elements of $X$, we have that each $Y_{i}$ is a subset of $T_{fin}(X)$ 
supported by $supp(X)$, and the family $(Y_{i})_{i \in \mathbb{N}}$ is an 
infinite, uniformly supported, subset of $\wp_{fs}(T_{fin}(X))$.

2. The proof is actually the same as in the above item since every $X_{i} \in \wp_{fs}(\wp_{fs}(X))$.

3. Suppose, by contradiction, that $X \times Y$ is FSM uniformly infinite.  
Thus, there exists an infinite injective family $((x_{i},y_{i}))_{i \in I} 
\subseteq X \times Y$ and a finite $S \subseteq A$ with the property that 
$supp((x_{i}, y_{i})) \subseteq S$ for all $i \in I$ (1). Fix some $j \in 
I$. We claim that $supp((x_{j}, y_{j})) =supp(x_{j}) \cup supp(y_{j})$.  
Let $U=(x_{j}, y_{j})$, and $S=supp(x_{j})\cup supp(y_{j})$. Obviously, $S$ 
supports $U$. Indeed, let us consider $\pi\in Fix(S)$. We have that $\pi\in 
Fix(supp(x_{j}))$ and also $\pi\in Fix(supp(y_{j}))$ Therefore, $\pi\cdot 
x_{j}=x_{j}$ and $\pi\cdot y_{j}=y_{j}$, and so $\pi \otimes (x_{j}, 
y_{j})=(\pi \cdot x_{j}, \pi \cdot y_{j})=(x_{j}, y_{j})$, where $\otimes$ 
represent the $S_{A}$ action on $X \times Y$ described in Proposition 
\ref{p1}. Thus, $supp(U) \subseteq S$. It remains to prove that $S 
\subseteq supp(U)$. Fix $\pi \in Fix(supp(U))$. Since $supp(U)$ supports 
$U$, we have $\pi \otimes (x_{j}, y_{j})=(x_{j}, y_{j})$, and so $(\pi 
\cdot x_{j}, \pi \cdot y_{j})=(x_{j}, y_{j})$, from which we get $\pi \cdot 
x_{j}=x_{j}$ and $\pi \cdot y_{j}= y_{j}$. Thus, $supp(x_{j}) \subseteq 
supp(U)$ and $supp(y_{j}) \subseteq supp(U)$. Hence $S=supp(x_{j})\cup 
supp(y_{j}) \subseteq supp(U)$.

According to relation (1) we obtain, $supp(x_{i})\cup supp(y_{i}) \subseteq 
S$ for all $i \in I$. Thus, $supp(x_{i}) \subseteq S$ for all $i \in I$ and 
$supp(y_{i}) \subseteq S$ for all $i \in I$ (2). Since the family 
$((x_{i},y_{i}))_{i \in I}$ is infinite and injective, then at least one of 
the uniformly supported families $(x_{i})_{i \in I}$ and $(y_{i})_{i \in 
I}$ is infinite, a contradiction.

4.  Suppose, by contradiction, that $X + Y$ is FSM uniformly infinite. 
Thus, there exists an infinite injective family $(z_{i})_{i \in I} 
\subseteq X\times Y$ and a finite $S \subseteq A$ such that $supp(z_{i}) 
\subseteq S$ for all $i \in I$.  According to the construction of the 
disjoint union of two $S_{A}$-sets (see Proposition \ref{p1}), there should 
exist an infinite family of $(z_{i})_{i}$ of form $((0, x_{j}))_{x_{j} \in 
X}$ which is uniformly supported by $S$, or an infinite family of form 
$((1, y_{k}))_{y_{k} \in Y}$ which is uniformly supported by $S$.  Since 
$0$ and $1$ are constants, this means there should exist at least an 
infinite uniformly supported family of elements from $X$, or an infinite 
uniformly supported family of elements from $Y$, a contradiction. 
\end{proof}

The following result represents a significant extension of Theorem 2 in 
\cite{AFML} since we are able to prove that $\wp_{fs}(A)^{A}_{fs}$ does not 
contain an infinite uniformly supported subset (an so, neither one of its 
subsets such as $S_{A}$ or $A^{A}_{fs}$ does not contain an infinite 
uniformly supported subset).

\begin{theorem} \label{ti2}
All the sets presented below are FSM non-uniformly infinite (i.e. none of them contains infinite uniformly supported subsets).
\begin{enumerate}
\item The invariant set $A$ of atoms.
\item The powerset $\wp_{fs}(A)$ of the set of atoms.
\item The set $T_{fin}(A)$ of all finite injective tuples of atoms.
\item  The invariant set of all finitely supported functions $f:A \to \wp_{fs}(A)$.
\item The invariant set $A^{A}_{fs}$ of all finitely supported functions from $A$ to $A$.
\item The invariant set of all finitely supported functions $f:A \to A^{n}$, where $n \in \mathbb{N}$ and $A^{n}$ is the $n$-times Cartesian product of $A$.
\item  The invariant set of all finitely supported functions $f:A \to T_{fin}(A)$.
\item The sets
 $\wp_{fin}(A)$, $\wp_{cofin}(A)$, $\wp_{fin}(\wp_{fs}(A))$, or $\wp_{fin}(A^{A}_{fs})$.
\item Any construction of finite powersets of the following forms
$\wp_{fin}(\ldots \wp_{fin}(A))$, $\wp_{fin}(\ldots \wp_{fin}(A^{A}_{fs}))$, or $\wp_{fin}(\ldots \wp_{fin}(\wp_{fs}(A)))$.
\item Every finite Cartesian combination between the set $A$, $\wp_{fin}(A)$, $\wp_{cofin}(A)$, $\wp_{fs}(A)$ and $A^{A}_{fs}$.
\item The disjoint unions $A+A^{A}_{fs}$, $A+\wp_{fs}(A)$,  $\wp_{fs}(A)+A^{A}_{fs}$ and $A+\wp_{fs}(A)+A^{A}_{fs}$ and all finite disjoint unions between  $A$, $A^{A}_{fs}$ and  $\wp_{fs}(A)$.
\end{enumerate}
\end{theorem}

\begin{proof}
1. $A$ does not contain an infinite uniformly supported subset since for any finite set $S \subseteq A$ there are at most $|S|$ atoms supported by $S$, namely the elements of $S$.

2. $\wp_{fs}(A)$ does not contain an infinite uniformly supported subset since for any finite set $S \subseteq A$ there are at most $2^{|S|+1}$ subsets 
of $A$ supported by a certain finite set $S \subseteq A$, namely the 
subsets of $S$ and the supersets of $A \setminus S$.

3. $T_{fin}(A)$ does not contain an infinite uniformly supported subset 
because the finite injective tuples of atoms supported by a finite set $S$ are 
only those injective tuples formed by elements of $S$, being at most 
$1+A_{|S|}^{1}+A_{|S|}^{2}+\ldots+A_{|S|}^{|S|}$ such tuples, where 
$A_{n}^{k}=n(n-1)\ldots (n-k+1)$.

4. We prove that $\wp_{fs}(A)^{A}_{fs}$ does not contain 
infinite uniformly supported subsets.

We remark that if $S=\{s_{1},\ldots,s_{n}\}$ is a finite subset 
of an invariant set $(X, \cdot)$ containing no infinite uniformly 
supported subset, then $X^{S}_{fs}$ does not contain an infinite uniformly 
supported subset. For this we claim that there is an injection 
$\varphi$ from~$X^{S}_{fs}$ into~$X^{|S|}$ defined by: if $f \in 
X^{S}_{fs}$, then $\varphi(f)=(f(s_{1}),\ldots, f(s_{n}))$; if $\pi$ 
fixes $supp(s_{1}) \cup \ldots \cup supp(s_{n})$ pointwise, then $\varphi(\pi 
\widetilde{\star} f)=((\pi \widetilde{\star} f)(s_{1}),\ldots,(\pi 
\widetilde{\star} f)(s_{n}))=(\pi \cdot f(\pi^{-1} \cdot s_{1}),\ldots, 
\pi \cdot f(\pi^{-1} \cdot s_{n}))=(\pi \cdot f( s_{1}),\ldots, \pi \cdot 
f(s_{n}))$ $=\pi \otimes \varphi(f)$ for all $f \in X^{S}_{fs}$, where 
$\otimes$ is the $S_{A}$-action on~$X^{|S|}$, and $\widetilde{\star}$ is the canonical action on $X^{S}_{fs}$. Therefore~$\varphi$ is finitely supported. Obviously, $X^{|S|}$ 
does not contain an infinite uniformly supported subset; otherwise $X$ 
should contain itself an infinite uniformly supported subset.

Let us fix $n \in \mathbb{N}$. Assume, by contradiction, that there exist 
infinitely many functions $g:A \to \wp_{n}(A)$ (where $\wp_{n}(A)$ is the 
invariant set of all $n$-sized subsets of $A$) supported by the same 
finite set $S' \subseteq A$. Each $S'$-supported function $g:A \to 
\wp_{n}(A)$ can be uniquely decomposed into two $S'$-supported functions 
$g|_{S'}$ and $g|_{A \setminus S'}$ (this follows since both $S'$ and $A \setminus S'$ are supported by $S'$). Since there exist only finitely many functions from $S'$ to 
$\wp_{n}(A)$ supported by $S'$, there should exist infinitely many 
functions $g:(A\setminus S') \to \wp_{n}(A)$ supported by $S'$. For such a 
function $g$, let us fix an element $a\in A \setminus S'$. For each $\pi$ fixing $S' \cup \{a\}$ pointwise we have $\pi \star g(a)=g(\pi(a))=g(a)$ which 
means that $g(a)$ is supported by $S' \cup \{a\}$. Since $g(a)$ is an $n$-sized 
(i.e. finite) subset of atoms, we have 
$g(a)=supp(g(a)) \subseteq S' \cup \{a\}$. We distinguish two cases. In 
the first case, $g(a)=\{a,x_{2},\ldots, x_{n}\}$ with $x_{2}, \ldots, 
x_{n} \in S'$. Let $b$ be an arbitrary element from $A\setminus S'$, and 
so $(a\, b)$ fixes $S'$ pointwise, which means $g(b)=g((a\,b)(a))=(a\,b) \star 
g(a)=(a\,b) \star \{a, x_{2}, \ldots, x_{n}\}=\{(a\,b)(a), 
(a\,b)(x_{2}),\ldots, (a\,b)(x_{n})\}=\{b,x_{2},\ldots, x_{n}\}$. Thus, 
only the choice of $x_{2}, \ldots, x_{n}$ provides the distinction between 
$g$'s.  Since $S'$ is finite, $\{x_{2}, \ldots, x_{n}\}$ can be selected 
in $C_{|S'|}^{n-1}$ ways if $|S'|\geq n-1$, or in $0$ ways otherwise. In 
the second case we have $g(a)=\{x_{1},\ldots, x_{n}\}$ with $x_{1}, 
\ldots, x_{n} \in S'$.  For all $b \in A \setminus S$ we have that $(a\,b)$ fixes $S'$ pointwise, and so $g(b)=g((a\,b)(a))=(a\,b) \star g(a)=(a\,b) \star 
\{x_{1}, \ldots, x_{n}\}=\{x_{1},\ldots, x_{n}\}$. Since $S'$ is finite, 
$\{x_{1}, \ldots, x_{n}\}$ can be selected in $C_{|S'|}^{n}$ ways if
$|S'|\geq n$, or in~$0$ ways otherwise. In both cases, $g$'s can be 
defined only in finitely many ways.

We proved that there exist at most finitely many functions from $A$ to 
$\wp_{n}(A)$ supported by the same set of atoms. Let us assume by 
contradiction that $\wp_{fin}(A)^{A}$ contains an infinite $S$-uniformly 
supported subset. If $f:A \to \wp_{fin}(A)$ is a function supported by 
$S$, then we have $|f(a)|=|(a\,b)\star f(a)|=|f((a\,b)(a))|=|f(b)|$ for 
all $a,b \notin S$.  As above, each $S$-supported function $f:A \to 
\wp_{fin}(A)$ is uniquely decomposed into two $S$-supported functions 
$f|_{S}$ and $f|_{A \setminus S}$. However $f(A\setminus S) \subseteq 
\wp_{n}(A)$ for some $n \in \mathbb{N}$. We also know that there are at 
most finitely many $S$-supported functions from $S$ to $\wp_{fin}(A)$. 
Furthermore, there exist at most finitely many $S$-supported functions 
from $A\setminus S$ to $\wp_{n}(A)$ for each fixed $n \in \mathbb{N}$.  
Therefore, it should exist an infinite subset $M \subseteq \mathbb{N}$ 
such that we have at least one $S$-supported function $f:A\setminus S \to 
\wp_{k}(A)$ for any $k \in M$. Fix $a \in A \setminus S$. For each of the 
above $f$'s (that form an $S$-uniformly supported family $\mathcal{F}$) 
we have that $f(a)$'s form an uniformly supported family (by $S \cup 
\{a\}$) of $\wp_{fin}(A)$. If $S \cup \{a\}$ has $l$ elements, there 
exists a fixed $m\in M$ with $m>l$. However, $f(a)$ for a function 
$f:A\setminus S \to \wp_{m}(A)$ from $\mathcal{F}$, which is an $m$-sized 
subset of atoms cannot be supported by $S \cup \{a\}$ whose cardinality 
is less than~$m$. Therefore, the set of all $f(a)$'s cannot be infinite 
and uniformly supported.

Since there exists the empty supported bijection $X \mapsto A\setminus X$ 
from $\wp_{fin}(A)$ onto $\wp_{cofin}(A)$, we also have that there exist 
at most finitely many $S$-supported functions from $A$ to 
$\wp_{cofin}(A)$. Assume, by contradiction, that $\wp_{fs}(A)^{A}$ 
contains an infinite $S$-uniformly supported subset.  If $h:A \to 
\wp_{fs}(A)$ is a function supported by $S$, then consider $h(a)=X$ for 
some $a \in A \setminus S$. For $b \in A \setminus S$ we have 
$h(b)=(a\,b)\star X$, which means $h(A \setminus S)$ is formed only by 
finite subsets of atoms if $X$ is finite, and $h(A \setminus S)$ is formed 
only by cofinite subsets of atoms if $X$ is cofinite. However, we have at 
most finitely many $S$-supported functions from $S$ to $\wp_{fs}(A)$. 
Furthermore, we have at most finitely many $S$-supported functions from 
$A\setminus S$ to $\wp_{fin}(A)$, and at most finitely many $S$-supported 
functions from $A\setminus S$ to $\wp_{cofin}(A)$. We get a contradiction, 
and we conclude that $\wp_{fs}(A)^{A}_{fs}$ does not contain an infinite 
uniformly supported subset. 

5. There is an equivariant injection from $A^{A}_{fs}$ into $\wp_{fs}(A)^{A}_{fs}$, and the result is immediate.

6. There is an equivariant bijection between $(A^{n})^{A}_{fs}$ and 
$(A^{A}_{fs})^{n}$ defined as below. If $f:A \to A^{n}$ is a finitely 
supported function with $f(a)=(a_{1},\ldots, a_{n})$, we associate to $f$ 
the Cartesian pair $(f_{1},\ldots, f_{n})$ where for each $i \in \mathbb{N}$, 
$f_{i}:A \to A$ is defined by $f_{i}(a)=a_{i}$ for all $a \in  A$. Since 
$A^{A}_{fs}$ does not contain an infinite uniformly supported subset, neither 
$(A^{A}_{fs})^{n}$ contains an infinite uniformly supported subset.

7. Assume by contradiction that $T_{fin}(A)^{A}$ contains an infinite 
$S$-uniformly supported subset. If $f:A \to T_{fin}(A)$ is a function 
supported by $S$, then consider $f(a)=x$ for some $a \notin S$. For $b 
\notin S$ we have that $(a\,b)$ fixes $S$ pointwise, and so 
$f(b)=f((a\,b)(a))=(a\,b)\otimes f(a)=(a\,b)\otimes x$ which means 
$|f(a)|=|f(b)|$ for all $a,b \notin S$.  Each $S$-supported function $f:A 
\to T_{fin}(A)$ can be uniquely decomposed into two $S$-supported functions 
$f|_{S}$ and $f|_{A \setminus S}$. However $f(A\setminus S) \subseteq 
A'^{n}$ for some $n \in \mathbb{N}$, where $A'^{n}$ is the set of all 
injective $n$-tuples of $A$. We have at most finitely many $S$-supported 
functions from $S$ to $T_{fin}(A)$ (since $T_{fin}(A)^{S}$ cannot contain 
an infinite uniformly supported subset; otherwise $T_{fin}(A)$ would itself 
contain an infinite uniformly supported subset). Since $A'^{n}$ is a subset 
of $A^{n}$ and $A \setminus S$ is a subset of $A$, we have at most finitely 
many $S$-supported functions from $A\setminus S$ to $A'^{n}$ for each fixed 
$n \in \mathbb{N}$.  Therefore, there should exist an infinite subset $M 
\subseteq \mathbb{N}$ such that we have at least one $S$-supported function 
$g:A\setminus S \to A'^{k}$ for any $k \in M$. Fix $a \in A \setminus S$. 
For each of the above $g$'s (that form an $S$-supported family 
$\mathcal{F}$) we have that $g(a)$'s form an uniformly supported family (by 
$S \cup \{a\}$) of $T_{fin}(A)$, which is also infinite because tuples 
having different cardinalities are different and $M$ is infinite. We thus 
obtained a contradiction.

Items 8,9,10,11 follow from the above items involving Theorem \ref{ti1}
\end{proof}

\begin{remark}
Despite of Theorem \ref{ti2}(3), it is worth noting that the set 
$T^{\delta}_{fin}(A)=\underset{n\in \mathbb{N}}{\cup}A^{n}$ of all finite 
tuples of atoms (not necessarily injective) is FSM uniformly infinite. This 
follows as below. Fix $a \in A$ and $i \in \mathbb{N}$. We consider the 
tuple $x_{i}=(a,\ldots,a) \in A^{i}$. Clearly, $x_{i}$ is supported by 
$\{a\}$ for each $i \in \mathbb{N}$, and so $(x_{n})_{n \in \mathbb{N}}$ is 
a uniformly supported subset of $T^{\delta}_{fin}(A)$. 
\end{remark}

\begin{theorem} \label{propro} \ 
\begin{enumerate}
\item Let $X$ be a finitely supported subset of an invariant set.  
If $X$ is not FSM uniformly infinite, then each finitely supported 
injective mapping $f:X \to X$ should be surjective.

\item Let $X$ be a finitely supported subset of an invariant set.  
If $\wp_{fs}(X)$ is not FSM uniformly infinite, then each finitely 
supported surjective mapping $f:X \to X$ should be injective. 
The converse does not hold since every finitely supported surjective 
mapping $f:\wp_{fin}(A) \to \wp_{fin}(A)$ is also injective, while 
$\wp_{fs}(\wp_{fin}(A))$ is FSM uniformly infinite.
\end{enumerate}
\end{theorem} 

\begin{proof}
1. Assume, by contradiction, that $f: X \rightarrow X$ is a finitely 
supported injection with the property that $Im(f) \subsetneq X$. This means 
that there exists $x_{0}\in X$ such that $x_{0}\notin Im(f)$.  We can form 
a sequence of elements from~$X$ which has the first term $x_{0}$ and the 
general term $x_{n+1}=f(x_{n})$ for all $n \in \mathbb{N}$.  Since 
$x_{0}\notin Im(f)$ it follows that $x_{0} \neq f(x_{0})$. Since $f$ is 
injective and $x_{0} \notin Im(f)$, by induction we obtain that 
$f^{n}(x_{0}) \neq f^{m}(x_{0})$ for all $n,m \in \mathbb{N}$ with $n \neq 
m$. Furthermore, $x_{n+1}$ is supported by $supp(f)\cup supp(x_{n})$ for 
all $n \in \mathbb{N}$. Indeed, let $\pi \in Fix(supp(f)\cup supp(x_{n}))$. 
According to Proposition~\ref{2.18'}, $\pi \cdot x_{n+1}= \pi \cdot 
f(x_{n})=f(\pi \cdot x_{n})=f(x_{n})=x_{n+1}$. Since $supp(x_{n+1})$ is the 
least set supporting $x_{n+1}$, we obtain $supp(x_{n+1}) \subseteq 
supp(f)\cup supp(x_{n})$ for all $n \in \mathbb{N}$. By induction, we have 
$supp(x_{n}) \subseteq supp(f)\cup supp(x_{0})$ for all $n\in\mathbb{N}$. 
Thus, all $x_{n}$ are supported by the same set of atoms $supp(f)\cup 
supp(x_{0})$, which means the family $(x_{n})_{n \in \mathbb{N}}$ is 
infinite and uniformly supported, contradicting the hypothesis.

2. Let $f: X \to X$ be a finitely supported surjection. Since $f$ is 
surjective, we can define the function $g:\wp_{fs}(X) \rightarrow 
\wp_{fs}(X)$ by $g(Y) = f^{-1}(Y)$ for all $Y\in \wp_{fs}(X)$ which is 
finitely supported by $supp(f) \cup supp(X)$ (according to the $S$-finite 
support principle) and injective. Alternatively, we can provide a direct 
proof that $g$ is finitely supported. Let $Y$ be an arbitrary element from 
$\wp_{fs}(X)$.  We claim that $f^{-1}(Y) \in \wp_{fs}(X)$. Let $\pi$ fix 
$supp(f) \cup supp(Y) \cup supp(X)$ pointwise, and $y \in f^{-1}(Y) $. This 
means $f(y) \in Y$. Since $\pi$ fixes $supp(f)$ pointwise and $supp(f)$ 
supports $f$, we have $f(\pi \cdot y)= \pi \cdot f(y) \in \pi \star Y = Y$, 
and so $\pi \cdot y \in f^{-1}(Y)$. Therefore, $f^{-1}(Y)$ is finitely 
supported, and so the function $g$ is well defined. We claim that $g$ is 
supported by $supp(f) \cup supp(X)$. Let $\pi$ fix $supp(f) \cup supp(X)$ 
pointwise.  For any arbitrary $Y \in \wp_{fs}(X)$ we get $\pi \star Y \in 
\wp_{fs}(X)$ and $\pi \star g(Y) \in \wp_{fs}(X)$. Furthermore, $\pi^{-1}$ 
fixes $supp(f)$ pointwise, and so $f(\pi^{-1} \cdot x)=\pi^{-1} \cdot f(x)$ 
for all $x\in X$. For any arbitrary $Y \in \wp_{fs}(X)$, we have that $z 
\in g(\pi \star Y) = f^{-1}(\pi \star Y) \Leftrightarrow f(z) \in \pi \star 
Y \Leftrightarrow \pi^{-1} \cdot f(z) \in Y \Leftrightarrow f(\pi^{-1} 
\cdot z) \in Y \Leftrightarrow \pi^{-1} \cdot z \in f^{-1}(Y) 
\Leftrightarrow z \in \pi \star f^{-1}(Y)=\pi \star g(Y)$. If follows that 
$g(\pi \star Y)=\pi \star g(Y)$ for all $Y \in \wp_{fs}(X)$, and so $g$ is 
finitely supported. Now, since $\wp_{fs}(X)$ is not FSM uniformly infinite, 
it follows from item 1 that $g$ is surjective.

Now let us consider two elements $a,b \in X$ such that $f(a)=f(b)$. We 
prove by contradiction that $a=b$. Suppose that $a \neq b$. Let us consider 
$Y=\{a\}$ and $Z=\{b\}$. Obviously, $Y,Z \in \wp_{fs}(X)$. Since $g$ is 
surjective, for $Y$ and $Z$ there is $Y_{1}, Z_{1} \in \wp_{fs}(X)$ such 
that $f^{-1}(Y_{1})=g(Y_{1})=Y$ and~$f^{-1}(Z_{1})=g(Z_{1})=Z$. We know 
that $f(Y) \cap f(Z)= \{f(a)\}$. Thus, $f(a) \in f(Y)=f(f^{-1}(Y_{1})) 
\subseteq Y_{1}$. Similarly, $f(a) =f(b) \in f(Z)=f(f^{-1}(Z_{1})) 
\subseteq Z_{1}$, and so $f(a) \in Y_{1} \cap Z_{1}$.  Thus, $a \in 
f^{-1}(Y_{1} \cap Z_{1})=f^{-1}(Y_{1}) \cap f^{-1}(Z_{1})=Y \cap Z$. 
However, since we assumed that $a \neq b$, we have that $Y \cap Z = 
\emptyset$, which represents a contradiction. It follows that $a=b$, and so 
$f$ is injective.

In order to prove the invalidity of the reverse implication, we prove 
that any finitely supported surjective mapping $f:\wp_{fin}(A) \to 
\wp_{fin}(A)$ is also injective, while $\wp_{fs}(\wp_{fin}(A))$ is FSM 
uniformly infinite (since it contains an infinite uniformly supported 
countable subset $(X_{n})_{n \in \mathbb{N}}$ where, for any $n \in 
\mathbb{N}$, $X_{n}$ is defined as the equivariant set of all $n$-sized 
subsets of atoms). Let us consider a finitely supported surjection 
$f:\wp_{fin}(A) \to \wp_{fin}(A)$. Let $X \in \wp_{fin}(A)$. Then 
$supp(X)=X$ and $supp(f(X))=f(X)$. Since $supp(f)$ supports $f$ and 
$supp(X)$ supports $X$, for any $\pi$ fixing pointwise $supp(f) \cup 
supp(X)=supp(f) \cup X$ we have $\pi \star f(X)=f(\pi \star X)=f(X)$ which 
means $supp(f) \cup X$ supports $f(X)$, that is $f(X)=supp(f(X)) \subseteq 
supp(f) \cup X$ (claim 1).

For a fixed $m \geq 1$, let us fix $m$ (arbitrarily chosen) atoms $b_{1}, 
\ldots, b_{m} \in A \setminus supp(f)$. Let us consider 
$\mathcal{U}=\{\{a_{1},\ldots, a_{n}, b_{1},\ldots, b_{m}\} \,|\, 
a_{1},\ldots, a_{n} \in supp(f) , n \geq 1\} \cup \{\{ b_{1},\ldots, 
b_{m}\}\}$. The set $\mathcal{U}$ is finite since $supp(f)$ is finite and 
$b_{1}, \ldots, b_{m} \in A \setminus supp(f)$ are fixed. Let us consider 
$Y \in \mathcal{U}$, that is $Y \setminus supp(f)=\{b_{1},\ldots, b_{m}\}$. 
There exists $Z \in \wp_{fin}(A)$ such that $f(Z)=Y$. According to (claim 
1), $Z$ must be either of form $Z=\{c_{1},\ldots, c_{k}, b_{i_{1}},\ldots, 
b_{i_{l}}\}$ with $c_{1},\ldots, c_{k} \in supp(f)$ and $b_{i_{1}},\ldots, 
b_{i_{l}} \in A \setminus supp(f)$ or of form $Z=\{ b_{i_{1}},\ldots, 
b_{i_{l}}\}$ with $b_{i_{1}},\ldots, b_{i_{l}} \in A \setminus supp(f)$. In 
both cases we have $\{b_{1},\ldots, b_{m}\} \subseteq \{b_{i_{1}},\ldots, 
b_{i_{l}}\}$. We should prove that $l=m$. Assume, by contradiction, that 
there exists $b_{i_{j}}$ with $j \in \{1,\ldots,l\}$ such that $b_{i_{j}} 
\notin \{b_{1},\ldots, b_{m}\}$. Then $(b_{i_{j}}\; b_{1}) \star Z=Z$ since 
both $b_{i_{j}}, b_{1} \in Z$ and $Z$ is a finite subset of $A$ 
($b_{i_{j}}$ and~$b_{1}$ are interchanged in~$Z$ under the effect of the 
transposition $(b_{i_{j}}\; b_{1})$, while the other atoms belonging to~$Z$ 
are left unchanged, meaning that the whole $Z$ is left invariant under 
$\star$). Furthermore, since $b_{i_{j}}, b_{1} \notin supp(f)$ we have that 
$(b_{i_{j}}\; b_{1})$ fixes $supp(f)$ pointwise, and, because $supp(f)$ 
supports $f$, we get $f(Z)=f((b_{i_{j}}\; b_{1}) \star Z)=(b_{i_{j}}\; 
b_{1}) \star f(Z)$ which is a contradiction because $b_{1} \in f(Z)$ while 
$b_{i_{j}} \notin f(Z)$. Thus, $\{b_{i_{1}},\ldots, b_{i_{l}}\} = 
\{b_{1},\ldots, b_{m}\}$, and so $Z \in \mathcal{U}$. Therefore, 
$\mathcal{U} \subseteq f(\mathcal{U})$ which means $|\mathcal{U}| \leq 
|f(\mathcal{U})|$. However, since $f$ is a function and $\mathcal{U}$ is 
finite, we get $|f(\mathcal{U})|\leq |\mathcal{U}|$. We obtain 
$|\mathcal{U}|=|f(\mathcal{U})|$ and, because $\mathcal{U}$ is finite with 
$\mathcal{U} \subseteq f(\mathcal{U})$, we get $\mathcal{U}=f(\mathcal{U})$ 
(claim 2) which means that $f|_{\mathcal{U}}:\mathcal{U} \to \mathcal{U}$ 
is surjective. Since $\mathcal{U}$ is finite, $f|_{\mathcal{U}}$ should be 
injective, i.e. $f(U_{1}) \neq f(U_{2})$ whenever $U_{1}, U_{2} \in 
\mathcal{U}$ with $U_{1}\neq U_{2}$~(claim~3).

Whenever $d_{1}, \ldots, d_{v} \in A \setminus supp(f)$ 
with$\{d_{1},\ldots, d_{v}\}\neq \{b_{1},\ldots, b_{m}\}$, $v \geq 1$, and 
considering $\mathcal{V}=\{\{a_{1},\ldots, a_{n}, d_{1},\ldots, d_{v}\} 
\,|\, a_{1},\ldots, a_{n} \in supp(f) , n \geq 1\} \cup \{\{ d_{1},\ldots, 
d_{v}\}\}$, we conclude that $\mathcal{U}$ and $\mathcal{V}$ are disjoint. 
Whenever $U_{1} \in \mathcal{U}$ and $V_{1} \in \mathcal{V}$ we have 
$f(U_{1}) \in \mathcal{U}$ and $f(V_{1}) \in \mathcal{V}$ by using the same 
arguments used to prove (claim 2), and so $f(U_{1})\neq f(V_{1})$ (claim 
4). If $\mathcal{T}=\{\{a_{1},\ldots, a_{n}\} \,|\, a_{1},\ldots, a_{n} \in 
supp(f)\}$ and $Y \in \mathcal{T}$, then there is $T'\in \wp_{fin}(A)$ such 
that $Y=f(T')$. Similarly as in (claim 2), we should have $T' \in 
\mathcal{T}$. Otherwise, if $T'$ belongs to some $\mathcal{V}$ considered 
above, i.e. if $T'$ contains an element outside $supp(f)$, we get the 
contradiction $Y=f(T') \in \mathcal{V}$) and so $\mathcal{T} \subseteq 
f(\mathcal{T})$ from which $\mathcal{T}=f(\mathcal{T})$ since $\mathcal{T}$ 
is finite (using similar arguments as those involved to prove (claim 3) 
from $\mathcal{U} \subseteq f(\mathcal{U})$). Thus, 
$f|_{\mathcal{T}}:\mathcal{T} \to \mathcal{T}$ is surjective. Since 
$\mathcal{T}$ is finite, $f|_{\mathcal{T}}$ should be also injective, 
namely $f(T_{1}) \neq f(T_{2})$ whenever $T_{1}, T_{2} \in \mathcal{T}$ 
with $T_{1}\neq T_{2}$ (claim 5). The case $supp(f)=\emptyset$ is contained 
in the above analysis; it leads to $f(\emptyset)=\emptyset$ and $f(X)=X$ 
for all $X \in \wp_{fin}(A)$. We also have $f(T_{1}) \neq f(V_{1})$ 
whenever $T_{1} \in \mathcal{T}$ and $V_{1} \in \mathcal{V}$ since 
$f(T_{1}) \in \mathcal{T}$, $f(V_{1}) \in \mathcal{V}$ and $\mathcal{T}$ 
and $\mathcal{V}$ are disjoint (claim 6). Since $b_{1}, \ldots, b_{m}$ and 
$d_{1}, \ldots, d_{v}$ were arbitrarily chosen from $A \setminus supp(f)$, 
the injectivity of $f$ leads from the claims (3), (4), (5) and (6) covering 
all the possible cases for two different finite subsets of atoms and 
comparison of the values of~$f$ over the related subsets of atoms.  
\end{proof}

Theorem \ref{propro} (related to Theorem 2 in \cite{AFML}) allows us to 
establish a strong result generalizing the approach in \cite{AFML} 
by claiming that a finitely supported mapping $f : \wp_{fin}(A) \to 
\wp_{fin}(A)$ is injective \emph{if and only if} it is surjective.

\begin{theorem} \label{Tti} 
Let $X$ be a finitely supported subset of an invariant set $(Z, \cdot)$. 
If $X$ contains an infinite, finitely supported, totally ordered subset, 
then it is FSM uniformly infinite. 
\end{theorem}

\begin{proof}
Assume that $X$ contains an infinite, finitely supported, totally ordered 
subset $(Y, \leq)$.  We claim that $Y$ is uniformly supported by 
$supp(\leq) \cup supp(Y)$. Let $\pi$ be a permutation fixing $supp(\leq) 
\cup supp(Y)$ pointwise and let $y \in Y$ an arbitrary element. Since $\pi$ 
fixes $supp(Y)$ pointwise and $supp(Y)$ supports $Y$, we obtain that $\pi 
\cdot y \in Y$, and so we should have either $y<\pi \cdot y$, or $y=\pi 
\cdot y$, or $\pi \cdot y<y$. If $y<\pi \cdot y$, then, because $\pi$ fixes 
$supp(\leq)$ pointwise and because the mapping $z\mapsto \pi \cdot z$ is 
bijective from $Y$ to $\pi \star Y$, we get $y<\pi \cdot y<\pi^{2} \cdot 
y<\ldots < \pi^{n} \cdot y$ for all $n \in \mathbb{N}$. However, since any 
permutation of atoms interchanges only finitely many atoms, it has a finite 
order in the group $S_{A}$, and so there is $m \in \mathbb{N}$ such that 
$\pi^{m}=Id$ (where $Id$ is the identity on $A$). This means $\pi^{m} \cdot 
y=y$, and so we get $y<y$ which is a contradiction. Similarly, the 
assumption $\pi \cdot y<y$, leads to the relation $\pi^{n} \cdot 
y<\ldots<\pi \cdot y<y$ for all $n \in \mathbb{N}$ which is also a 
contradiction since $\pi$ has finite order. Therefore, $\pi \cdot y=y$, and 
because $y$ was arbitrary chosen form $Y$, $Y$ should be a uniformly 
supported infinite subset of $X$. 
\end{proof}

\begin{definition} \label{dd1} \ 
\begin{itemize}
 \item Two FSM sets $X$ and $Y$ are \emph{FSM equipollent} if there exists 
a finitely supported bijection $f:X\to Y$.  
\item The \emph{FSM cardinality} of  $X$ is defined as the equivalence class of all FSM sets 
equipollent to $X$, and is denoted by $|X|$. 
\end{itemize}
\end{definition}
According to Definition \ref{dd1} for two FSM sets $X$ and $Y$, we have 
$|X|=|Y|$ if and only if there exists a finitely supported bijection $f:X 
\to Y$. On the family of cardinalities we can define the relations:
\begin{itemize}
\item  $\leq$ by: $|X| \leq |Y|$ if and only if there is a finitely supported
injective (one-to-one) mapping $f:X \to Y$.
\item $\leq^{*}$ by: $|X| \leq^{*} |Y|$ if and only if there is a finitely 
supported surjective (onto) mapping $f:Y \to X$. 
\end{itemize}

By using Theorem 4.5 and Theorem 4.6 from \cite{mvlsc}, 
we can present the following result.

\begin{theorem} \label{cardord} \ 
\begin{enumerate}
\item The relation $\leq$ is equivariant, reflexive, anti-symmetric and transitive, but it is not total.
\item The relation $\leq^{*}$ is equivariant, reflexive and transitive, but it is not anti-symmetric, nor total.
\end{enumerate}
\end{theorem}

\begin{theorem} \label{TTRR} 
Let $X$ be a finitely supported subset of an invariant set $(Y, \cdot)$
\begin{enumerate}
\item If $|X|=|X \times X|$, then $|X|=2|X|$. The converse does not hold.
\item If $|X|=2|X|$, then $X$ is FSM uniformly infinite. 
The converse does not hold.
\end{enumerate}
\end{theorem}

\begin{proof} 
1. Fix two elements $x_{1}, x_{2} \in X$ with $x_{1}\neq x_{2}$. We can 
define an injection $f:X \times \{0,1\} \to X \times X$ by $f(u)=\left\{ 
\begin{array}{ll} (x,x_{1}) & \text{for}\: u=(x,0)\\ (x,x_{2}) & 
\text{for}\: u=(x,1) \end{array}\right.$. Clearly, by checking the 
condition in Proposition \ref{2.18'} and using Proposition \ref{p1}, we 
have that $f$ is supported by $supp(X) \cup supp(x_{1}) \cup supp(x_{2})$ 
(since $\{0,1\}$ is necessarily a trivial invariant set), and so 
$|X\times\{0,1\}|\leq |X \times X|$. Thus, $|X\times\{0,1\}| \leq |X|$. 
Obviously, there is an injection $i: X \to X\times\{0,1\}$ defined by 
$i(x)=(x,0)$ for all $x \in X$ which is supported by $supp(X)$. 
According to Theorem \ref{cardord}, we get $2|X|=|X \times \{0,1\}|=|X|$.

Let us consider $Z=\mathbb{N} \times A$. We make the remark that 
$|\mathbb{N}\times \mathbb{N}|=|\mathbb{N}|$ by considering the equivariant 
injection $h:\mathbb{N} \times \mathbb{N} \to \mathbb{N}$ defined by 
$h(m,n)=2^{m}3^{n}$ and using Theorem \ref{cardord}. Similarly, 
$|\{0,1\}\times \mathbb{N}|=|\mathbb{N}|$ by considering the equivariant 
injection $h':\mathbb{N} \times \{0,1\}\to \mathbb{N}$ defined by 
$h'(n,0)=2^{n}$ and $h'(n,1)=3^{n}$ and using Theorem \ref{cardord}. We 
have $2|Z|=2|\mathbb{N}||A|=|\mathbb{N}||A|=|Z|$. However, we prove that 
$|Z \times Z|\neq |Z|$. Assume the contrary, and so we have $|\mathbb{N} 
\times (A \times A)|=|\mathbb{N} \times A \times \mathbb{N} \times 
A|=|\mathbb{N} \times A|$. Thus, there is a finitely supported injection 
$g: A \times A \to \mathbb{N} \times A$, and so there is a finitely 
supported surjection $f:\mathbb{N} \times A \to A \times A$ defined as 
$f(y)=\left\{ \begin{array}{ll} g^{-1}(y), & \text{if}\: \text{$y \in 
Im(g)$ }\\ x_{0}, & \text{if}\: \text{$y \notin Im(g)$}\: 
\end{array}\right.$ where $x_{0}$ is a fixed element in $A \times A$.  Let 
us consider three different atoms $a,b,c \notin supp(f)$.  There exists 
$(i,x) \in \mathbb{N} \times A$ such that $f(i,x)=(a,b)$. Since $(a\,b) \in 
Fix(supp(f))$ and $\mathbb{N}$ is trivial invariant set, we have 
$f(i,(a\,b)(x))=(a\,b)f(i,x)=(a\,b)(a,b)=((a\,b)(a),(a\,b)(b))=(b,a)$. We 
should have $x=a$ or $x=b$, otherwise $f$ is not a function. Assume without 
losing the generality that $x=a$, which means $f(i,a)=(a,b)$. Therefore 
$f(i,b)=f(i,(a\,b)(a))=(a\,b)f(i,a)=(a\,b)(a,b)=(b,a)$. Similarly, since 
$(a\,c),(b\,c) \in Fix(supp(f))$, we have 
$f(i,c)=f(i,(a\,c)(a))=(a\,c)f(i,a)=(a\,c)(a,b)=(c,b)$ and 
$f(i,b)=f(i,(b\,c)(c))=(b\,c)f(i,c)=(b\,c)(c,b)=(b,c)$. But $f(i,b)=(b,a)$ 
contradicting the functionality of $f$.

2. Let us consider an element $y_{1}$ belonging to an invariant set (whose 
action is also denoted by $\cdot$) with $y_{1}\notin X$ (such an element 
can be a non-empty element in $\wp_{fs}(X) \setminus X$, for instance). 
Fix $y_{2} \in X$.  One can define a mapping $f:X \cup \{y_{1}\} \to X \times 
\{0,1\}$ by $f(x)=\left\{ \begin{array}{ll} (x,0) & \text{for}\: x \in X\\ 
(y_{2}, 1) & \text{for}\: x=y_{1} \end{array}\right.$. Clearly, $f$ is 
injective and it is supported by $S=supp(X) \cup supp(y_{1}) \cup 
supp(y_{2})$ because for all $\pi$ fixing $S$ pointwise we have $f(\pi 
\cdot x)=\pi \cdot f(x)$ for all $x \in X \cup \{y_{1}\}$. Therefore, $|X 
\cup \{y_{1}\}| \leq |X \times \{0,1\}|=|X|$, and so there is a finitely 
supported injection $g:X \cup \{y_{1}\} \to X$. The mapping $h:X \to X$ 
defined by $h(x)=g(x)$ is injective, supported by $supp(g) \cup supp(X)$, 
and $g(y_{1}) \in X \setminus h(X)$, which means $h$ is not surjective. 
According to Theorem \ref{propro}(1), $X$ should be FSM uniformly infinite.

Let us denote $Z=A \cup \mathbb{N}$. Since $A$ and $\mathbb{N}$ are 
disjoint, we have that $Z$ is an invariant set. Clearly, $Z$ is FSM 
uniformly infinite. Assume, by contradiction, that $|Z|=2|Z|$, that is $|A 
\cup \mathbb{N}|=|A+A+\mathbb{N}|=|(\{0,1\}\times A) \cup \mathbb{N}|$. 
Thus, there is a finitely supported injection $f':(\{0,1\}\times A) \cup 
\mathbb{N}\to A \cup \mathbb{N}$, and so there exists a finitely supported 
injection $f:(\{0,1\}\times A) \to A \cup \mathbb{N}$. We prove that 
whenever $\varphi:A \to A \cup \mathbb{N}$ is finitely supported and 
injective, we have $\varphi(a) \in A$ for $a \notin supp(\varphi)$. 
Let us assume by contradiction that there is $a \notin supp(\varphi)$ such that 
$\varphi(a)\in \mathbb{N}$. Since $supp(\varphi)$ is finite, there exists 
$b \notin supp(\varphi)$, $b \neq a$. Thus, $(a\,b)$ fixes $supp(\varphi)$ 
pointwise, and so $\varphi(b)=\varphi((a\,b)(a))=(a\,b)\diamond 
\varphi(a)=\varphi(a)$ since $(\mathbb{N}, \diamond)$ is a trivial 
invariant set. This contradicts the injectivity of $\varphi$.  We can 
consider the mappings $\varphi_{1},\varphi_{2}: A \to A \cup \mathbb{N}$ 
defined by $\varphi_{1}(a)=f(0,a)$ for all $a \in A$ and 
$\varphi_{2}(a)=f(1,a)$ for all $a \in A$, that are injective and supported 
by $supp(f)$. Therefore, $f(\{0\} \times A)=\varphi_{1}(A)$ contains at 
most finitely many element from $\mathbb{N}$, and $f(\{1\} \times 
A)=\varphi_{2}(A)$ also contains at most finitely many element from 
$\mathbb{N}$. Thus, $f$ is an injection from $(\{0,1\}\times A)$ to $A \cup 
T$ where $T$ is a finite subset of $\mathbb{N}$. It follows that $f(\{0\} 
\times A)$ contains an infinite finitely supported subset of atoms $U$, and 
$f(\{1\} \times A)$ contains an infinite finitely supported subset of atoms 
$V$. Since $f$ is injective, it follows that $U$ and $V$ are infinite 
disjoint finitely supported subsets of $A$, which contradicts the fact that 
any subset of $A$ is either finite or cofinite.
\end{proof}

\section{Conclusion}

The newly developed theory of finitely supported sets allows the 
computational study of structures which are very large, possibly infinite, 
but containing enough symmetries such that they can be clearly/concisely 
represented and manipulated. Uniformly supported sets are particularly of 
interest because they involve boundedness properties of supports, meaning 
that the support of each element in an uniformly supported set is contained 
in the same finite set of atoms. In this way, all the individuals in an 
infinite uniformly supported family can be characterized by involving only 
finitely many characteristics.

In this paper we described FSM uniformly infinite sets that are finitely 
supported sets containing infinite, uniformly supported subsets.  Firstly 
we proved that the finite powerset and the uniform powerset of a set that 
is FSM uniformly finite is also FSM non-uniformly infinite (Theorem 
\ref{lem1} and Theorem \ref{lem2}). Finitely supported order-preserving 
self-mappings on the finite powerset and, respectively, on the uniform 
powerset of a set that is FSM non-uniformly infinite have least fixed 
points (Theorem~\ref{th1}). This is an important extension of Tarski's 
fixed point theorem for complete lattices that is specific to FSM; 
generally, order-preserving functions on finite powersets do not have fixed 
points since the finite powersets are not complete lattices. Particularly, 
finitely supported order-preserving mappings $f:\wp_{fin}(A) \to \wp_{fin}(A)$, 
finitely supported order-preserving mappings $f:\wp_{fin}(\wp_{fs}(A)) \to 
\wp_{fin}(\wp_{fs}(A))$ and finitely supported order-preserving mappings 
$f:\wp_{fin}(A^{A}_{fs}) \to \wp_{fin}(A^{A}_{fs})$ should have least fixed 
points that are supported by $supp(f)$ in each case.  
Another fixed point property is described in Theorem~\ref{th11}. 
Particularly, finitely supported progressive (inflationary) self-mappings 
defined on $\wp_{fin}(A)$ have infinitely many fixed points as proved in 
Proposition \ref{pf1ty}. We can also prove that any finitely supported, 
strict order-preserving, self-mapping $f$ on $\wp_{fin}(A)$ has infinitely many 
fixed points (namely all the sets $X \setminus supp(f)$ with $X \in \wp_{fin}(A)$).

Operations with FSM uniformly (in)finite sets are presented in Theorem 
\ref{ti1}. We were able to prove that $A$, $\wp_{fs}(A)$, $T_{fin}(A)$, 
$\wp_{fin}(\wp_{fs}(A))$, $A^{A}_{fs}$, $\wp_{fin}(A^{A}_{fs})$, 
$(A^{n})^{A}_{fs}$ (for a fixed $n \in \mathbb{N}$), $T_{fin}(A)^{A}_{fs}$ 
and $\wp_{fs}(A)^{A}_{fs}$ are FSM non-uniformly infinite, while 
$\wp_{fs}(\wp_{fin}(A))$ and $T^{\delta}_{fin}(A)$ are FSM uniformly 
infinite. Connections between FSM uniformly non-infinity and 
injectivity/surjectivity of self-mappings on FSM sets are presented in 
Theorem \ref{propro}. One can easily remark that a finitely supported 
function $f:A \to A$ is injective if and only if it is surjective. 
Furthermore, any finitely supported injection $f:\wp_{fs}(A) \rightarrow 
\wp_{fs}(A)$ is also surjective, any finitely supported injection 
$f:\wp_{fin}(\wp_{fs}(A)) \rightarrow \wp_{fin}(\wp_{fs}(A))$ is also 
surjective, and any finitely supported injection $f:A^{A}_{fs} \to 
A^{A}_{fs}$ is also surjective. These results generalize/extend related 
results presented in Theorem 2 of \cite{AFML}.
In Theorem \ref{Tti} we proved that a finitely supported subset of an 
invariant set containing an infinite, finitely supported, totally ordered 
subset is FSM uniformly infinite. Finally, we connected the concept of being 
FSM uniformly infinite with cardinality properties of form $|X|=|X \times 
X|$ and $|X|=2|X|$, respectively (Theorem~\ref{TTRR}).

The case study presented in this paper can be significantly extended by 
presenting several other definitions of infinity (Dedekind type, 
Mostowski type, Tarski type and Kuratowski type), and then comparing them 
in the framework of atomic finitely supported sets. 
This is the topic of a future~paper.

\nocite{*}
\bibliographystyle{eptcs}

\end{document}